\documentclass[11pt]{amsart}
\usepackage{amssymb,amsmath,amsthm}
\usepackage{a4}
\usepackage{epsf}
\usepackage{epsfig}
\usepackage{url}
\begin{document}

\newcommand{\NP}{$\mathcal{N}\mathcal{P}$}
\newcommand{\newsymb}{{\mathcal P}}
\newcommand{\Pn}{{\mathcal P}^n}
\newcommand{\Kn}{{\mathcal K}^n}
\newcommand{\R}{{\mathbb R}}
\newcommand{\N}{{\mathbb N}}
\newcommand{\Q}{{\mathbb Q}}
\newcommand{\Z}{{\mathbb Z}}
\newcommand{\C}{{\mathbb C}}

\newcommand{\enorm}[1]{\Vert #1\Vert}
\newcommand{\inter}{\mathrm{int}}
\newcommand{\conv}{\mathrm{conv}}
\newcommand{\aff}{\mathrm{aff}}
\newcommand{\lin}{\mathrm{lin}}
\newcommand{\cone}{\mathrm{cone}}
\newcommand{\bd}{\mathrm{bd}}

\newcommand{\dist}{\mathrm{dist}}
\newcommand{\trans}{\intercal}
\newcommand{\diam}{\mathrm{diam}}
\newcommand{\vol}{\mathrm{vol}}
\newcommand{\F}{\mathrm{F}}
\newcommand{\W}{\mathrm{W}}
\newcommand{\V}{\mathrm{V}}

\newcommand{\LE}{\mathrm{G}}
\newcommand{\lE}{\mathrm{g}}
\newcommand{\sa}{\mathrm{a}}

\newcommand{\pp}{\mathfrak{p}}
\newcommand{\pf}{\mathfrak{f}}
\newcommand{\pg}{\mathfrak{g}}
\newcommand{\PP}{\mathfrak{P}}
\newcommand{\pl}{\mathfrak{l}}
\newcommand{\pv}{\mathfrak{v}}
\newcommand{\cl}{\mathrm{cl}}
\newcommand{\bx}{\overline{x}}

\def\ip(#1,#2){#1\cdot#2}

\newtheorem{theorem}{Theorem}[section]
\newtheorem{theorem*}{Theorem}
\newtheorem{corollary}[theorem]{Corollary}
\newtheorem{lemma}[theorem]{Lemma}
\newtheorem{remark}[theorem]{Remark}
\newtheorem{definition}[theorem]{Definition}  
\newtheorem{conjecture}{Conjecture}[section] 
\newtheorem{proposition}[theorem]{Proposition}  
\newtheorem{claim}[theorem]{Claim}
\newtheorem{problem}[theorem]{Problem}
\numberwithin{equation}{section}

\title[A Blichfeldt-type inequality]{A Blichfeldt-type inequality for the \\ surface area}   
%\date{\today}
%\thanks{}
\author{Martin Henk}
\address{Martin Henk, Universit\"at Magdeburg, Institut f\"ur Algebra und Geometrie,
  Universit\"ats\-platz 2, D-39106 Magdeburg, Germany}
\email{henk@math.uni-magdeburg.de}
\author{J\"org M.~Wills}
\address{J\"org M.~Wills, Universit\"at Siegen,  Mathematisches Institut,
 ENC, D-57068 Siegen, Germany}
\email{wills@mathematik.uni-siegen.de}
%\dedicatory{Dedicated to Uli Betke on the occasion of his 60th birthday}


\begin{abstract} 
In 1921 Blichfeldt gave an upper bound on the number of integral  points  contained in a convex body  in terms of the volume of the body. More precisely, he showed that $\#(K\cap\Z^n)\leq n!\,\vol(K)+n$, whenever $K\subset\R^n$ is a convex body containing $n+1$ affinely  independent integral points. Here we prove an analogous inequality with respect to the surface area $\F(K)$, namely $ \#(K\cap\Z^n) < \vol(K) + ((\sqrt{n}+1)/2 )\,(n-1)!\,\F(K)$. The proof is based on a  slight improvement of Blichfeldt's bound in the case when $K$ is a non-lattice translate of a lattice  polytope, i.e., $K=t+P$, where $t\in\R^n\setminus\Z^n$ and $P$ is an $n$-dimensional polytope with integral vertices. Then we have $\#((t+P)\cap\Z^n)\leq n!\,\vol(P)$.    

Moreover, in the $3$-dimensional case we prove a stronger inequality, namely $\#(K\cap\Z^n) < \vol(K) + 2\,\F(K)$.
\end{abstract}

\maketitle

\section{Introduction}
Let $\Kn$ be the set of all convex bodies in the $n$-dimensional Euclidean space $\R^n$.  
For a subset $S\subset \R^n$ and the integral lattice $\Z^n$ let
$\LE(S)$ be the lattice point enumerator of $S$, i.e., $\LE(S)=\#(S\cap
\Z^n)$. By $\vol(S)$ we denote, as usual, the volume, i.e.,
the $n$-dimensional Lebesgue measure, of $S$. 

The problem
to bound $\LE(K)$, $K\in\Kn$,  in terms of continuous functionals, as
e.g.~the intrinsic volumes, has a long history in convexity
(cf.~e.g.~\cite{BeckRobins2007,  Gritzmann1993c, Lagarias1995}), and the first general upper bound with respect to the
volume is due to Blichfeldt \cite{Blichfeldt1920/21}
\begin{equation} 
   \LE(K) \leq n!\,\vol(K)+n,
\label{eq:blichfeldt}
\end{equation} 
provided $\dim(K\cap\Z^n)=n$, i.e., $K$ contains $n+1$ affinely
independent lattice points of $\Z^n$. This bound is best possible for
any number of lattice points, as, for instance, the simplex 
$S_k=\conv\{0,k\,e_1,\dots,e_n\}$,  
$k\in\N$, shows. Here $e_i$ denotes the $i$-th canonical unit vector and so we
have $\LE(S_k)=k+n$ and $\vol(S_k)=k/n!$. 
Our main result is an inequality analogous to \eqref{eq:blichfeldt},
but now with respect to the surface area $\F(K)$ of the body.
\begin{theorem} Let $K\in\Kn$ with $\dim(K\cap\Z^n)=n$.
 Then 
\begin{equation*}
  \LE(K)< \vol(K)+\frac{\sqrt{n}+1}{2}\,(n-1)!\,\F(K).
\end{equation*} 
\label{thm:main}
\end{theorem} 
In contrast to Blichfeldt's inequality the volume is now weighted by the factor $1$ instead of $n!$, which is apparently  best possible. 
 We conjecture that the factor
 $\frac{\sqrt{n}+1}{2}$ can be omitted in this inequality. In dimension $2$ this follows easily from Pick's identity \cite[pp.~38]{BeckRobins2007}, and the $3$-dimensional case is settled in the next theorem. 
\begin{theorem} Let $K\in\mathcal{K}^3$ with $\dim(K\cap\Z^3)=3$. Then
\begin{equation*}
   \LE(K)< \vol(K) + 2\,\F(K).
\end{equation*}  
\label{thm:three_case}
\end{theorem} 
We remark that an inequality of the form $\LE(K)<\vol(K)+(n-1)!\,\F(K)$ 
would be tight in the sense that $(n-1)!$ in front of the surface area can not be replaced by $c\,(n-1)!$ for a constant $c<1$.   
  To see this we note that for the simplex $S_1$ with $n+1$ lattice points we have $\F(S_1)=(n+\sqrt{n})/(n-1)!$.

The inequality in Theorem \ref{thm:main} may also be regarded as a
counterpart to a well-known lower bound  on $\LE(K)$ due to Bokowski,
Hadwiger and Wills \cite{Bokowski1972}. They proved that 
\begin{equation}
  \vol(K)- \frac{1}{2}\F(K) < \LE(K),
\label{eq:bhw_lattice}
\end{equation} 
and this inequality is best possible. 

The proof of Theorem \ref{thm:main} is based on a lemma on lattice
points in a translate of a lattice polytope.  To this end we denote by $\Pn\subset \Kn$ the set of all lattice polytopes,
i.e., polytopes having integral vertices.  
\begin{lemma} Let $P\in\Pn$ with $\dim(P\cap\Z^n)=n$, and let
  $t\in\R^n\setminus\Z^n$. Then 
\begin{equation*}
   \LE(t+P)\leq n!\vol(P),
\end{equation*} 
and the inequality is best possible for any number of lattice points.
\label{lem:main}
\end{lemma} 
In other words, if we have a non-lattice translate of $P$ then we can slightly 
improve Blichfeldt's bound \eqref{eq:blichfeldt} by $n$. This does not
mean, however, that $t+P$ has less lattice points than $P$. For
instance, for $n>2$ and $m\in\N$ let $T_m$ be the so called Reeve simplex 
$T_m=\conv\{0,e_1,\dots,e_{n-1},m\,v\}$, where $v=e_1+e_2+\cdots+e_n$.  Then the vertices are the only lattice points in $T_m$, but $(1/2)\,v+T_m$ contains the $m$ lattice points $v,2\,v,\dots,m\,v$. In the $2$-dimensional case the situation is different and for a detailed discussion of
lattice points in translates of lattice polygons  we refer to
\cite{Hadwiger1976}. 

Since \eqref{eq:blichfeldt} and the inequality in Lemma \ref{lem:main}
depend only on the volume, it is easy  to generalize them to an
arbitrary lattice $\Lambda\subset\R^n$ with determinant
$\det\Lambda>0$. Then, with the setting as before,  we have 
\begin{equation}
\mathrm{i)}\,\,  \#(K\cap\Lambda)\leq n!\frac{\vol(K)}{\det\Lambda}+n,
\text{ and } \mathrm{ii)}\,\, \#((t+P)\cap\Lambda)\leq n!\frac{\vol(P)}{\det\Lambda}.
\label{eq:blichfeldt_translate} 
\end{equation} 
In the case of Theorem \ref{thm:main} we conjecture
that the right statement for general lattices is 
\begin{conjecture}Let $\Lambda\subset\R^n$ be a lattice and let $K\in\Kn$ with $\dim(K\cap\Lambda)=n$. Then 
\begin{equation*}
  \#(K\cap\Lambda)< \frac{\vol(K)}{\det\Lambda}+(n-1)!\,\frac{\F(K)}{\det\Lambda_{n-1}},
\end{equation*} 
where $\det\Lambda_{n-1}$ is the minimal determinant of an
$(n-1)$-dimensional sublattice of  $\Lambda$.
\end{conjecture} 
In the $2$-dimensional case the correctness of the inequality is again an easy consequence of Pick's identity.
It is also not hard to verify such an inequality with an additional factor
of order $n$ in front of the surface area, and we will give an outline of a proof of
this result in the last section (see Corollary \ref{cor:general}). 
For a corresponding conjecture
regarding the lower bound \eqref{eq:bhw_lattice} we refer to \cite{ Schnell1992, Schnell1991}.  In Section 1 we will prove  Lemma \ref{lem:main} and Theorem \ref{thm:main}. The proof of Theorem \ref{thm:three_case} is based on results on the inner/outer parallel body of a convex body and is given in the second section.

\section{Proof of Lemma \ref{lem:main} and Theorem \ref{thm:main}}

The proof of Lemma  \ref{lem:main} will be an immediate consequence of the fact that for $n$ linearly independent lattice points $a_1,\dots,a_n\in\Z^n$ and the associated half-open 
parallelepiped $C=\{\sum_{i=1}^n \rho_i\,a_i : 0\leq \rho_i<1 \}$ one has 
\begin{equation}
 \LE(C)=|\det(a_1,\dots,a_n)|.
\label{eq:parallelepiped}
\end{equation}
Observe, both sides just describe the index of the sublattice generated by $a_1,\dots,a_n$ with respect to $\Z^n$ (see e.g.~\cite[p.~22]{Gruber1987}).

\begin{proof}[Proof of Lemma \ref{lem:main}] Let $P\subset\R^n$ be a lattice polytope and let $t\in\R^n\setminus\Z^n$. Let $S_1,\dots,S_m\subseteq P$ be $n$-dimensional lattice simplices such that $P=\cup_{i=1}^m S_i$ and $\dim(S_i\cap S_j)\leq n-1$ for $i\ne j$. For instance, we can take any lattice triangulation of $P$. Then $\vol(P)=\sum_{i=1}^m \vol(S_i)$ and 
\begin{equation*} 
   \LE(t+P)\leq \sum_{i=1}^m \LE(t+S_i).
\end{equation*}   
Hence it suffices to prove the statement for an $n$-dimensional lattice simplex $S$, say. Without loss of generality let $0,a_1,\dots,a_n$ be the vertices of $S$, $a_i\in\Z^n$, and let $C$  be the half-open parallelepiped generated by $a_1,\dots,a_n$. Then by \eqref{eq:parallelepiped} we have 
\begin{equation} 
\LE(C) = |\det(a_1,a_2,\dots,a_n)|=n!\,\vol(S).  
\label{eq:index}
\end{equation}
Next we observe that for any vector $\bar{t}\in\R^n$ 
\begin{equation}
\label{eq:translate}
   \LE(\bar{t}+C)=\LE(C).
\end{equation}
This is a well-known fact, but for sake of completeness we give a short argument: Let $\bar{t}=\sum_{i=1}^n \tau_i\,a_i$. For $\bar{b}=\bar{t}+\sum_{i=1}^n \rho_i\,a_i \in (\bar{t}+C)\cap\Z^n$ the vector 
$f(\bar{b})$ defined by $f(\bar{b})=\sum_{i=1}^n (\tau_i+\rho_i -\lfloor\tau_i+\rho_i\rfloor)\,a_i$ is contained in $C\cap\Z^n$. Here $\lfloor x\rfloor$ denotes the largest integer not bigger then $x$. It is easy to see that $f$ is a bijection between $(\bar{t}+C)\cap\Z^n$ and $C\cap\Z^n$, and hence we have verified \eqref{eq:translate}.
% an injective map between  $(\bar{t}+C)\cap\Z^n$ and $C\cap\Z^n$. In order to show that it is surjective let $b=\sum_{i=1}^n \rho_i\,a_i\in C\cap\Z^n$ and let $\bar{b}=\sum_{i=1}^n\tau_i\,a_i+\sum_{i=1}^n(\lfloor\tau_i-\rho_i\rfloor-(\tau_i-\rho_i))\,a_i$. Then $\bar{b}\in (\bar{t}+C)\cap\Z^n$ and $f(\bar{b})=b$.

Finally, since 
\begin{equation}
 t+S \subset (t+C)\cup\{t+a_1,t+a_2,\ldots,t+a_n\},
\label{eq:inclusion}
\end{equation} 
we  get   $\LE(t+S) \leq \LE(t+C)$ for $t\in\R^n\setminus\Z^n$. Together with 
\eqref{eq:translate} and \eqref{eq:index} we obtain the desired inequality for $S$, 
and thus for the lattice polytope $P$. 

In order to show that it is best possible let $S_k$ be the simplex defined in the introduction. 
Then $k=n!\,\vol(S_k)$ and if we translate $S_k$ by $\frac{1}{2}e_1$, for instance, then $e_1,\dots,k\,e_1$ are the only lattice points in $\frac{1}{2}e_1+S_k$.
\end{proof} 

\begin{remark} If $t=0$ then \eqref{eq:inclusion} gives $\LE(S)\leq \LE(C)+n$ because $a_1,\dots,a_n$ are the only points in $S$ not contained in $C$. Thus we get by the same argument  Blichfeldt's  inequality $\LE(K)\leq n!\,\vol(K)+n$.  
\end{remark} 

For the proof of Theorem \ref{thm:main} we also need some facts about lattice polytopes. Let $P\subset\R^n$ be a lattice polytope. Then we can describe it as  
\begin{equation}
    P=\{x\in\R^n : a_i\cdot x \leq b_i,\,1\leq i\leq m\},
\label{eq:lattice_polytope} 
\end{equation}  
for some $a_i\in\Z^n$, $b_i\in\Z$. Here $x\cdot y$ denotes the inner product, and by $\enorm{\cdot}$ we denote the  associated  Euclidean norm. Without loss of generality let $F_i=P\cap\{x\in\R^n : a_i\cdot x =b_i\}$ be the facets of $P$, $1\leq i\leq m$. We may also assume that the vectors $a_i$ are primitive vectors, i.e., $\conv\{0,a_i\}\cap\Z^n=\{0,a_i\}$. In this case we have (cf.~e.g.~\cite[Proposition 1.2.9]{Martinet2003}) 
\begin{equation}
  \det(\aff F_i \cap\Z^n)=\enorm{a_i},
\label{eq:determinant}
\end{equation} 
where $\det(\aff F_i \cap\Z^n)$ is the determinant of the $(n-1)$-dimensional sublattice of $\Z^n$ contained in the affine hull of $F_i$.

\begin{proof}[Proof of Theorem \ref{thm:main}]  Let $K\in\Kn$ with $\dim(K\cap\Z^n)=n$. 
By the monotonicity of $\vol(\cdot)$ and $F(\cdot)$ it suffices to prove the conjecture for the $n$-dimensional lattice polytope $P=\conv\{K\cap\Z^n\}$. Let $C_n$ be the cube of edge length $1$ centered at the origin. Let $L_1=\{z\in P\cap\Z^n : z+C_n\subset P\}$ and 
$L_2=(P\cap \Z^n)\setminus L_1$. Obviously,  we have 
\begin{equation}
  \#L_1 \leq \vol(P), 
\label{eq:volumepart} 
\end{equation}   
and it remains to bound the size of the set $L_2$. To this end let $P$ be given as in \eqref{eq:lattice_polytope} with facets $F_1,\dots,F_m$. 
For each lattice point $z\in L_2$ there exists a facet $F_i$ such that $z+C_n$ intersects $F_i$, i.e., there exists an $x\in C_n$ with 
$a_i\cdot z+a_i\cdot x> b_i$. Hence we have 
\begin{equation*}
  a_i\cdot z> b_i-a_i\cdot x \geq b_i-\frac{1}{2}|a_i|,
\end{equation*} 
where $|\cdot|$ denotes the $l_1$-norm. Since the left hand side is an integer we obtain 
\begin{equation}
  a_i\cdot z\geq  b_i-\gamma_i\text{ with }  \gamma_i=\left\lceil \frac{1}{2}|a_i|\right\rceil-1.
\label{eq:lower_violated}
\end{equation}  
Thus 
\begin{equation}
  L_2\subset \bigcup_{i=1}^m (Q_i\cap\Z^n),
\label{eq:set_l2}
\end{equation} 
where $Q_i=\conv\{F_i,F_i -(\gamma_i/\enorm{a_i}^2)  a_i\}$ is the prism with basis $F_i$ and height $\gamma_i/\enorm{a_i}$ in the direction $-a_i$. Next we claim that for $1\leq i\leq m$ 
\begin{equation}
  \LE(Q_i) < \frac{\sqrt{n}+1}{2}\,(n-1)!\vol_{n-1}(F_i)+(n-1),
\label{eq:claim}   
\end{equation}
where $\vol_{n-1}(\cdot)$ denotes the $(n-1)$-dimensional volume.
Each lattice point in such a prism $Q_i$ is contained in one of the layers  
\begin{equation*}
  H_i(j)=Q_i\cap\{x\in \R^n : a_i\cdot x = b_i-j\},\quad j=0,1,\dots,\gamma_i. 
\end{equation*}  
Of course, $H_i(0)=F_i$ is an $(n-1)$-dimensional lattice polytope with respect to the lattice $\Lambda_{F_i}=\aff(F_i)\cap\Z^n$. On account of \eqref{eq:determinant}  we get from Blichfeldt's inequality (see \eqref{eq:blichfeldt_translate} i))    
\begin{equation}   
  \LE(H_i(0))=\#(F_i\cap\Lambda_{F_i}) \leq (n-1)!\frac{\vol_{n-1}(F_i)}{\enorm{a_i}}+(n-1).
\label{eq:first_layer}
\end{equation} 
Now let $v_i\in\Z^n$ be any lattice vector in the lattice hyperplane $\{x\in\R^n : a_i\cdot x=b_i-1\}$ and let $w_i$ be a lattice vector in $F_i$. Since 
\begin{equation*}
  H_i(j)=H_i(0)-j\,\frac{1}{\enorm{a_i}^2}\,a_i 
\end{equation*}  
we have 
\begin{equation*}
 H_i(j)\cap\Z^n = \left(H_i(0)-j\,\frac{1}{\enorm{a_i}^2}\,a_i + j(w_i-v_i)\right)\cap\Z^n=(j\,t_i+F_i)\cap\Z^n,
\end{equation*} 
with $t_i=w_i-v_i-1/\enorm{a_i}^2\,a_i\in\{x\in\R^n : a_i\cdot x=0\}$. Since $a_i$ is primitive and $j\leq \gamma_i<\enorm{a_i}^2$ we find  that $j\,t_i\in\R^n\setminus\Z^n$ for $1\leq j\leq \gamma_i$. Thus we may apply in these cases Lemma \ref{lem:main}, or more precisely \eqref{eq:blichfeldt_translate} ii), and obtain  
\begin{equation*}
 \LE(H_i(j)) \leq (n-1)!\frac{\vol_{n-1}(F_i)}{\enorm{a_i}}, \quad j=1,\dots, \gamma_i.
\end{equation*}  
Together with \eqref{eq:first_layer} we get 
\begin{equation*}
\begin{split}
\LE(Q_i) & = \LE(H_i(0)) + \sum_{j=1}^{\gamma_i} \LE(H_i(j)) 
 %                &\leq (n-1)!\frac{\vol_{n-1}(F_i)}{\enorm{a_i}}+(n-1)+ \gamma_i\left( (n-1)!\frac{\vol_{n-1}(F_i)}{\enorm{a_i}} \right) \\ 
 \leq  \frac{\lceil\frac{1}{2}|a_i|\rceil}{\enorm{a_i}} (n-1)!\,\vol_{n-1}(F_i) + (n-1)\\  
 &<   \frac{\sqrt{n}+1}{2} (n-1)!\,\vol_{n-1}(F_i) + (n-1), 
\end{split} 
\end{equation*} 
 and so we have verified \eqref{eq:claim}. 

Finally, in order to prove the inequality of the theorem  we have to consider the lattice points which we count more than once in the right hand side of \eqref{eq:set_l2}, and we claim 
\begin{equation}
   \#L_2\leq \sum_{i=1}^m \LE(Q_i)-m\,(n-1).
\label{eq:claim_l2}
\end{equation} 
To this end we consider the vertices $v_1,\dots,v_k$ of $P$. Let $g_{n-1}(v_j)$ be the number of facets containing $v_j$ and let $f_0(F_i)$ be the number of vertices of the facet $F_i$. Obviously, we have 
\begin{equation}
   \sum_{i=1}^m f_0(F_i) = \sum_{j=1}^k g_{n-1}(v_j). 
\label{eq:vertices_facets}  
\end{equation} 
Since each facet has at least $n$ vertices and each vertex is contained in at least $n$ facets we conclude from \eqref{eq:vertices_facets} that 
$\sum_{i=1}^m f_0(F_i) \geq \max\{m,k\}\,n \geq k + m\,(n-1)$. This shows 
\eqref{eq:claim_l2} and so, in view of \eqref{eq:claim} we get 
\begin{equation*}
  \#L_2< \frac{\sqrt{n}+1}{2}\,(n-1)!\,\F(P).
\end{equation*}  
Together with  \eqref{eq:volumepart} we obtain the desired inequality.
\end{proof}

\section{Proof of Theorem \ref{thm:three_case}} 

For the proof of Theorem \ref{thm:three_case} we need a bit of the theory of 
intrinsic volumes for which we refer to \cite{Schneider1993}. Let $K\in\Kn$ and let $B_n$ be the $n$-dimensional unit ball of volume $\kappa_n$.  The outer parallel body $K+\rho\,B_n$ of $K$ at distance $\rho$ is the Minkowski sum of $K$ and $\rho\,B_n$, i.e., $K+\rho\,B_n=\{x+y : x\in K,\,y\in\rho\,B_n\}$. Its volume can be described by the so called Steiner polynomial 
\begin{equation}
  \vol(K+\rho\,B_n) = \sum_{i=0}^n \V_i(K)\,\kappa_{n-i}\,\rho^{n-i}, 
\label{eq:steiner_polynomial}
\end{equation} 
where $V_i(K)$ is called the $i$-th intrinsic volume of $K$.  
In particular, we have $\V_n(K)=\vol(K)$, $\V_{n-1}(K)=(1/2)\F(K)$, and $\V_0(K)=1$. 
It was conjectured by Wills that  
\begin{equation}
  \LE(K) \leq \sum_{i=0}^n \V_i(K), 
\label{eq:wills}
\end{equation} 
but, in general, this inequality does not hold (see \cite{Betke1993a, Hadwiger1979}). In dimension three, however, it is true \cite{Overhagen1975} and so we have 
\begin{equation}
\LE(K) \leq \V_3(K)+ \V_2(K) + V_1(K) + 1.
\label{eq:overhagen}
\end{equation}  
The inner parallel body of $K$ at distance $\rho$ is given by the set 
\begin{equation*}
  K\ominus\rho\,B_n=\left\{x\in K : x+\rho\,B_n\subseteq K\right\}.
\end{equation*} 
If $K\ominus\rho\,B_n$ is non-empty then we trivially have 
$(K\ominus\rho\,B_n)+\rho\,B_n\subseteq K$.

\begin{proof}[Proof of Theorem \ref{thm:three_case}] Again it suffices to prove the inequality for the $3$-dimensional lattice polytope $P=\conv\{K\cap\Z^3\}$.  According to \eqref{eq:overhagen} and \eqref{eq:steiner_polynomial} we obtain 
\begin{equation*}
\begin{split}
  &\LE(P) \leq  \V_3(P)+\V_2(P) + \V_1(P)+1 \\
   &< \V_3(P) + \V_1(P)\kappa_1\frac{1}{\sqrt{\pi}} 
              +  \V_2(P)\kappa_2\left(\frac{1}{\sqrt{\pi}}\right)^2+\kappa_3\left(\frac{1}{\sqrt{\pi}}\right)^3+\left(1-\kappa_3\left(\frac{1}{\sqrt{\pi}}\right)^3\right)\\ 
&= \vol\left(P+\pi^{-1/2}\,B_3\right)+\left(1-\frac{4}{3\sqrt{\pi}}\right).
\end{split} 
\end{equation*}
Hence, if $P\ominus \pi^{-1/2}\,B_3\ne\emptyset$ we get 
\begin{equation*}
\LE(P\ominus \pi^{-1/2}\,B_3)< \vol((P\ominus \pi^{-1/2}\,B_3)+\pi^{-1/2}\,B_3)+0.25 \leq \vol(P)+0.25.
\label{eq:inner_lattice}
\end{equation*} 
On the other hand it was shown in  \cite[Korollar 1]{McMullen1973a} that  
\begin{equation*}
 \LE(P)-\LE(P\ominus 3^{-1/2}\,B_3) \leq \F(P) + 2. 
\end{equation*}  
Combining the last two inequalities yields  
\begin{equation*}
\begin{split}
 \LE(P) & \leq \LE(P\ominus 3^{-1/2}\,B_3)+\F(P)+2 \leq \LE(P\ominus \pi^{-1/2}\,B_3)+\F(P)+2 \\ & < \vol(P)+\F(P)+2.25.
\end{split} 
\end{equation*} 
Since the surface area of a $3$-dimensional lattice polytope is not less  than the surface area of the simplex $S_1$, which is equal to $(3+\sqrt{3})/2>2.25$, we finally obtain 
\begin{equation*}
 \LE(P)< \vol(P)+2\,\F(P).
\end{equation*}   
\end{proof} 

In the context with the conjectured inequality \eqref{eq:wills} it was shown by Bokowski \cite{Bokowski1975}  that for $n\leq 5$  
\begin{equation*}
  \LE(K)\leq \vol(K+\kappa_n^{-1/n}\,B_n). 
\end{equation*} 
With $\rho_n=\kappa_n^{-1/n}$ this leads, as in the proof above, to 
$\LE(P\ominus \rho_nB_n)\leq \vol(P)$ where $P=\conv\{K\cap\Z^n\}$. In order to estimate the remaining lattice points $\LE(P)-\LE(P\ominus \rho_n\,B_n)$, which are close to the boundary of $P$, we can proceed as in the proof of Theorem \ref{thm:main} where we bound the size of the set $L_2$. This leads, roughly speaking, for $n\leq 5$ to an inequality of the form 
\begin{equation*}
  \LE(K) < \vol(K) + \left(\rho_n+\frac{1}{2}\right)\,(n-1)!\,\F(K),  
\end{equation*}  
which is stronger than the one of Theorem \ref{thm:main}. Since the improvement, however, is marginal we omit a detailed proof.

\section{The inequality for arbitrary lattices} 
In order to present an inequality as in Theorem \ref{thm:main} for arbitrary lattices we need some basic facts and notions from Geometry of Numbers for which we refer to \cite{Gruber1987}. For a lattice $\Lambda \subset\R^n$ let $\Lambda^\star=\{y\in\R^n : y\cdot b \in\Z \text{ for all } b\in\Lambda\}$ be its polar lattice. The length (norm) of a shortest non-zero lattice vector in a lattice $\Lambda$ is denoted by $\lambda_1(\Lambda)$, and an $(n-1)$-dimensional sublattice of $\Lambda$ with minimal determinant is denoted by $\Lambda_{n-1}$. Then we have (cf.~e.g.~\cite[Proposition 1.2.9]{Martinet2003})
\begin{equation}
 \det\Lambda\cdot\lambda_1(\Lambda^\star) =\det\Lambda_{n-1}.
\label{eq:polar_lattice}
\end{equation}  
 Moreover, we need the so called Dirichlet-Voronoi cell 
$\mathrm{DV}(\Lambda)$ of a lattice $\Lambda$ consisting of all points whose nearest lattice point in $\Lambda$ is the origin, i.e., 
\begin{equation*} 
\mathrm{DV}(\Lambda)=\{x\in\R^n : \enorm{x}\leq\enorm{x-b} \text{ for all } b\in\Lambda\}.
\end{equation*} 
Then $\vol(\mathrm{DV}(\Lambda))=\det\Lambda$ and the smallest radius of  a ball containing $\mathrm{DV}(\Lambda)$ is called the inhomogeneous minimum of $\Lambda$ and will be denoted by $\mu(\Lambda)$. So in the case of the integral lattice $\Z^n$ the Dirichlet-Voronoi cell is just the cube of edge length 1 centered at the origin and $\mu(\Z^n)=\sqrt{n}/2$, $\lambda_1(\Z^n)=1$, $\det\Z^n_{n-1}=1$  and $(\Z^n)^\star=\Z^n$.

With these notations we can generalize Theorem \ref{thm:main} as follows. 
\begin{theorem} Let $\Lambda\subset\R^n$ be a lattice and  
let $K\in\Kn$ with $\dim(K\cap\Lambda)=n$.
 Then 
 \begin{equation*}
    \LE(K)\leq \frac{\vol(K)}{\det\Lambda}+
    \left(\mu(\Lambda)\lambda_1(\Lambda^\star)+1\right)\,(n-1)!\,
  \frac{\F(K)}{\det\Lambda_{n-1}}.
  \end{equation*} 
\label{thm:general}
\end{theorem} 
 
Observe, in the case $\Lambda=\Z^n$ we get essentially the inequality of Theorem \ref{thm:main}. By fundamental results of Banaszczyk \cite{Banaszczyk1993a} it is known that 
\begin{equation*}
 \mu(\Lambda)\lambda_1(\Lambda^\star) \leq c\,n,
\end{equation*} 
for some universal constant $c$ and so we have 
\begin{corollary}  Let $\Lambda\subset\R^n$ be a lattice and  
let $K\in\Kn$ with $\dim(K\cap\Lambda)=n$.
 Then 
 \begin{equation*}
    \LE(K)< \frac{\vol(K)}{\det\Lambda}+
    \left(c\,n\right)\,(n-1)!\,
  \frac{\F(K)}{\det\Lambda_{n-1}}
  \end{equation*} 
for some universal constant $c$.
\label{cor:general}
\end{corollary}   
  
Since the proof of Theorem \ref{thm:general} is just a simple adaption of the proof of Theorem \ref{thm:main} to this more general situation, we only give a sketch of it. 
\begin{proof}[Proof of Theorem \ref{thm:general}] We keep the notation of the proof of Theorem \ref{thm:main} and again we consider $P=\conv\{K\cap\Lambda\}$, which is a lattice polytope with respect to $\Lambda$. The outer normal vectors $a_i$ of the facets $F_i$ are now  lattice vectors of $\Lambda^\star$ and for the determinant of $\aff F_i\cap\Lambda$ (cf.~\eqref{eq:determinant}) we obtain (cf.~e.g.~\cite[Proposition 1.2.9]{Martinet2003})
\begin{equation}
 \det(\aff F_i\cap\Lambda) =\enorm{a_i}\det\Lambda.
\label{eq:facet_general}
\end{equation} 
The role of the cube $C_n$ is replaced by $\mathrm{DV}(\Lambda)$ and so we have 
\begin{equation*}
  \#L_1 \leq \frac{\vol(P)}{\det\Lambda}. 
\end{equation*}  
Now for $1\leq i\leq m$ let 
\begin{equation*} 
  \beta_i=\max\{a_i\cdot x : x\in \mathrm{DV}(\Lambda)\}.
\end{equation*} 
Then for each $z\in L_2$ there exists a facet $F_i$ with (cf.~\eqref{eq:lower_violated})
\begin{equation*}
  a_i\cdot z \geq b_i-\gamma_i \text{ with } \gamma_i=\lceil\beta_i\rceil-1. 
\end{equation*} 
Following the proof of Theorem \ref{thm:main} we obtain on account of \eqref{eq:facet_general}
\begin{equation*}
 \#(Q_i\cap\Lambda) \leq \frac{\lceil\beta_i\rceil}{\enorm{a_i}\det\Lambda}(n-1)!\,\vol_{n-1}(F_i)+(n-1).
\label{eq:qi_general}
\end{equation*} 
We remark that in the case  of an arbitrary lattice  one has to be a bit more careful when applying \eqref{eq:blichfeldt_translate} ii) to the single layers $H_i(j)$, since one has to ensure that the vector $j\,t_i$ is not a lattice vector of $\Lambda$ for $1\leq j\leq \gamma_i$. This follows, however, from  the definition of $\gamma_i$ and the definition of the Dirichlet-Voronoi cell. 
By the definitions of $\beta_i$ and of the inhomogeneous minimum, and by \eqref{eq:polar_lattice} we have 
\begin{equation*}
\begin{split} 
 \frac{\lceil\beta_i\rceil}{\enorm{a_i}\det\Lambda} & \leq \frac{\mu(\Lambda)\,\enorm{a_i}+1}{\enorm{a_i}\det\Lambda} = \frac{\mu(\Lambda)}{\det\Lambda} + \frac{1}{\enorm{a_i}\det\Lambda}\\
 & = \frac{\mu(\Lambda)\lambda_1(\Lambda^\star)}{\det\Lambda_{n-1}} + \frac{\lambda_1(\Lambda^\star)}{\enorm{a_i}}\frac{1}{\det\Lambda_{n-1}} \leq \frac{\mu(\Lambda)\lambda_1(\Lambda^\star)+1}{\det\Lambda_{n-1}}.
\end{split} 
\end{equation*} 
Finally, as in the proof of Theorem \ref{thm:main} we conclude 
\begin{equation*}
\begin{split} 
\#(P\cap\Lambda) &\leq \#L_1 + \#L_2 -m\,(n-1) \\ & \leq \frac{\vol(P)}{\det\Lambda}+(\mu(\Lambda)\lambda_1(\Lambda^\star)+1) (n-1)!\,\frac{\F(P)}{\det\Lambda_{n-1}}.
\end{split}   
\end{equation*} 
\end{proof}

%-------------------------------------------------------------------------

\noindent {\it Acknowledgements.}  The authors wish to thank Mar\'\i a A. Hern\'andez Cifre for valuable comments.

%---------------------------------bibliography----------------------------- 
%\bibliographystyle{amsplain} 
%\bibliography{martin's_literatur_database}
\def\cprime{$'$}
\providecommand{\bysame}{\leavevmode\hbox to3em{\hrulefill}\thinspace}
\providecommand{\MR}{\relax\ifhmode\unskip\space\fi MR }
% \MRhref is called by the amsart/book/proc definition of \MR.
\providecommand{\MRhref}[2]{%
  \href{http://www.ams.org/mathscinet-getitem?mr=#1}{#2}
}
\providecommand{\href}[2]{#2}

%--------------------------------------------------------------------------
\end{document}